\documentclass[letterpaper,11pt]{amsart}

\usepackage[divide={1in,*,1in}]{geometry}
\usepackage{amssymb,amsmath,amsthm}
\usepackage{enumitem}
\usepackage{color}
\usepackage[pdfauthor={Krzysztof K. Putyra},
						pdftitle={Separated presentations of modules over pullback rings},colorlinks]{hyperref}
\usepackage{xypic}
\usepackage{diagxy}

\def\clap#1{\hbox to 0pt{\hss#1\hss}}

\newtheorem{theorem}{Theorem}[section]
\newtheorem{lemma}[theorem]{Lemma}

\newtheorem{proposition}[theorem]{Proposition}
\newtheorem{corollary}[theorem]{Corollary}

\theoremstyle{definition}

\newtheorem{remark}[theorem]{Remark}
\newtheorem{definition}[theorem]{Definition}
\newtheorem{example}[theorem]{Example}

\definecolor{internalLink}{rgb}{0.5,0,0}
\definecolor{citeLink}{rgb}{0,0.5,0}
\definecolor{urlLink}{rgb}{0,0,0.5}

\hypersetup{linkcolor=internalLink}
\hypersetup{citecolor=citeLink}
\hypersetup{urlcolor=urlLink}

\DeclareMathOperator{\id}{id}		
\DeclareMathOperator{\im}{im}		
\DeclareMathOperator{\dom}{dom}	
\DeclareMathOperator{\coker}{coker}		

\newcommand*{\Z}{\mathbb{Z}}				

\newcommand{\quotient}[2]{\raisebox{0.9ex}{$#1$}\!\Big/\!\raisebox{-0.9ex}{$#2$}}

\makeatletter
\def\KhBracket{\@ifstar\KhBracketScaled\KhBracketSimple}
\makeatother

\newcommand*{\KhBracketScaled}[1]{\left\llbracket#1\right\rrbracket}
\newcommand*{\KhBracketSimple}[1]{\llbracket#1\rrbracket}

\let\diff\partial

\makeatletter

\def\arXiv{\@ifstar\arXiv@@\arXiv@}
\def\arXiv@#1{\href{http://front.math.ucdavis.edu/#1}{arXiv:#1}}
\def\arXiv@@#1#2{\href{http://front.math.ucdavis.edu/#2}{arXiv:#1}}

\makeatother

\makeatletter

\def\qnt#1{\,\overline{\!#1}}

\def\pbit#1{\pb@it#1,\@pbstop,,\@nil}
\def\pb@it#1,#2,#3,#4\@nil{
	\ifx\@pbstop#2\relax
		(#1_1\rightarrow\bar#1\leftarrow#1_2)%
	\else
		(#1\rightarrow#2\leftarrow#3)%
	\fi
}

\def\pb#1{\pb@d#1;\@pbstop;\@nil}
\def\pb@d#1;#2;#3\@nil{
	\ifx\@pbstop#2\relax
		\pb@d@nomaps#1,\@pbstop,,\@nil
	\else
		\pb@d@maps#1,\@pbstop,,;#2,\@pbstop,\@nil
	\fi
}
\def\pb@d@nomaps#1,#2,#3,#4\@nil{
	\ifx\@pbstop#2\relax
		\pb@@d#1_1,\qnt #1,#1_2;\@pbstop,,\@nil
	\else
		\pb@@d#1,#2,#3;\@pbstop,,\@nil
	\fi
}
\def\pb@d@maps#1,#2,#3,#4;#5\@nil{
	\ifx\@pbstop#2\relax
		\pb@@d#1_1,\qnt #1,#1_2;#5\@nil
	\else
		\pb@@d#1,#2,#3;#5\@nil
	\fi
}
\def\pb@@d#1,#2,#3;#4,#5,#6\@nil{
	\ifx\@pbstop#4\relax
		#1\longrightarrow#2\longleftarrow#3%
	\else\ifx\@pbstop#5\relax
		#1\stackrel{#4_1}\longrightarrow#2\stackrel{#4_2}\longleftarrow#3%
	\else
		#1\stackrel{#4}\longrightarrow#2\stackrel{#5}\longleftarrow#3%
	\fi\fi
}

\makeatother

\begin{document}

\title{Separated presentations of modules over pullback rings}

\author[Krzysztof K. Putyra]{Krzysztof K. Putyra}
\address{Department of Mathematics, Columbia University\\ New York, NY 10027}
\email{putyra@math.columbia.edu}

\begin{abstract}
	We define pullback and separated presentations of modules over pullback rings, and,
	if the~ring is a~pullback of epimorphisms over a~semisimple ring, an~algorithm
	reducing such a~presentation of a~module to an~$R$-diagram. The~latter is the~%
	input for a~classification algorithm of finitely generated modules over a~pullback
	ring of two Dedekind domains. As an~example we show how to obtain an~$R$-diagram
	for homology of a~chain complex of free modules over a~$p$-pullback ring.
\end{abstract}

\maketitle

\section{Introduction}\label{sec:intro}

In the~seminal paper \cite{Mixed ZG modules} Levy described an~algorithmic classification
of finitely generated modules over certain \emph{pullback rings}, defined as pullbacks of
diagrams of ring epimorphisms
\begin{equation}
	\pb{R;v},
\end{equation}
If each $R_i$ is a~Dedekind domain and $\qnt R$ is a~field, every finitely generated
module $M$ can be represented by a~collection of four homomorphisms
\begin{equation}
	\bfig
		\Atrianglepair/->``->`->`<-/<500,300>[%
			K`S_1`\qnt S`S_2;q_1``q_2`p_1`p_2
		]
	\efig
\end{equation}
called an~$R$-diagram for $M$, where $K$ and $\qnt S$ are finitely generated vector spaces
over $\qnt R$, each $S_i$ is a~finitely generated $R_i$-module, $p_i$ are $R_i$-linear
epimorphisms, and $q_i$ are $R_i$-linear monomorphisms, satisfying certain properties
(see Proposition~\ref{prop:R-diagrams}).
These diagrams can be faithfully translated into collections of four matrices over $\qnt R$,
and the~problem of determining whether two such collections represent isomorphic modules can
be solved algorithmically.

Unfortunately, in the~real life we usually do not have an~$R$-diagram of a~module. In~%
particular, a~chain complex of free modules provides only a~description of a~homology as
a~quotient of separated modules, which is far from the~input of the~Levy's classifying
algorithm. Our main motivation is to compute generalized Khovanov homology \cite{Khovanov,
ChCob}, which is defined over the~group ring $\Z[S_2]$. Therefore, we decided to broaden
the~class of admissible presentations, allowing to encode an~$R$-module as a~quotient
of any two separated modules. A~submodule of a~separated module is always separated,
(see Corollary~\ref{cor:separated-submodule}) which automatically gives us a~separated
presentation of homology modules, assuming chain modules are separated (a~free module
is always separated). The~main result of this paper is the~reduction procedure: it takes
as input any separated presentation of a~module $M$, and returns its $R$-diagram. It works
for every pullback ring $R$ as long as $\qnt R$ is semisimple. On a~side, we prove a~few
facts about separated modules and homomorphisms between them.

The~paper is organized as follows. We start with a~brief section on pullback rings,
introducing some notation, and showing how to recover a~pullback diagram from a~subdirect
sum description. Section~\ref{sec:separated} introduces pullback and separated modules,
and describes their basic properties: a~submodule of a~separated module is separated,
a~pullback description of a~separated module is unique, and that every homomorphism of
separated modules has a~unique pullback description. Morphisms between pullback modules
are analyzed in Section~\ref{sec:morphisms}. In particular, we characterize monomorphisms
in terms of their pullback descriptions, and we give a~few conditions for a~map to be an~%
epimorphism. These are used in Section~\ref{sec:general-modules}, in which we introduce
various presentations of a~generic $R$-module, and show how to reduce them to $R$-diagrams.
The~last section is devoted to computation of an~$R$-diagram for homology modules of
a~chain complex of modules over a~$p$-pullback ring, i.e. the~subring of $\Z\oplus\Z$
formed by pairs $(m,n)$ with $m\equiv n\ (\mathrm{mod}\,p)$.

\section{Pullback rings as subdirect summands}\label{sec:pullback-rings}

Let $R$ be a~pullback diagram defined by a~diagram of ring epimorphisms
\begin{equation}
	\pb{R;v}.
\end{equation}
It is easy to check that $R$ is a~subring of $R_1\oplus R_2$ consisting of elements
$(r_1,r_2)\in R_1\oplus R_2$ satisfying $v_1(r_1) = v_2(r_2)$. In~particular, every
module over $R_1\oplus R_2$ is also an~$R$-module. Following \cite{Mixed ZG modules} we
use the~notation $\pbit{r}$ for elements of $R$, where $\bar r = v_i(r_i)$. The ring
$R$ is a~\emph{subdirect sum} of $R_1\oplus R_2$ \cite{Subdirect sums}, i.e. each $R_i$
is an~image of $R$ under the~projection $R_1\oplus R_2\to R_i$. Conversely, every subdirect
sum of $R_1\oplus R_2$ is a~pullback ring, see \cite{Subdirect sums}. Here we show only how
to recover the~pullback diagram of $R$, regarded as a~subdirect sum of $R_1\oplus R_2$.

\begin{lemma}
	Consider a~pullback ring $R$ as a~subring of $R_1\oplus R_2$ and put $P_i=R\cap R_i$.
	Then $\qnt R = R / (P_1\oplus P_2)$.
\end{lemma}
\begin{proof}
	Because $P_1\cap P_2 = 0$, $P_1\oplus P_2$ is an~ideal of $R$. Clearly, $P_1 = (\ker v_1, 0)$
	and similarly for $P_2$. Hence, the~kernel of the~map
	\begin{equation}
		R\to \qnt R\qquad\qquad (r_1,r_2)\mapsto v_1(r_1) \quad \left(\ =v_2(r_2)\right)
	\end{equation}
	is precisely $P_1\oplus P_2$. Since the~map above is surjective, the~thesis follows.
\end{proof}

\section{Pullback and separated modules}\label{sec:separated}

Choose $R_i$-modules $S_i$ ($i=1,2$) and an $\qnt R$-module $\qnt S$. Because $v_i$ equips
the~latter with a~structure of an $R_i$-module, it makes sense to consider diagrams
\begin{equation}\label{diag:separated-module}
	\pb{S;p}
\end{equation}
where each $p_i$ is $R_i$-linear. Let $S$ be a~pullback of \eqref{diag:separated-module}
regarded as a~diagram of abelian groups. It is an~$R$-module with
\begin{equation}
	\pbit{r}\pbit{s} := \pbit{r_1s_1,\bar r\bar s,r_2s_2}.
\end{equation}

\begin{definition}\label{def:module-classes}
	A~module $S$ given as above is called a~\emph{pullback module} and the~diagram \eqref{diag:%
	separated-module} is a~\emph{pullback diagram} for $S$. We say the~diagram is \emph{preseparated}
	if each $p_i$ is surjective, and \emph{separated} if also $\ker p_i = P_iS_i$. Accordingly,
	the~module $S$ is called a~\emph{preseparated} or a~\emph{separated module}.
\end{definition}

\noindent
Pullback modules are precisely the~$R$-submodules of $S_1\oplus S_2$, and $S$ is a~subdirect
sum of $S_1\oplus S_2$ if and only if the~maps in $\pb{S}$ are surjective, i.e. when it is a~%
preseparated diagram for $S$. In fact, every $R$-submodule of $S_1\oplus S_2$ is a~separated
module. In~particular, this applies to pullback modules, showing that all three classes from
Definition~\ref{def:module-classes} are equal.

\begin{lemma}\label{lem:pullback-is-separated}
	For $i=1,2$ choose $R_i$-modules $T_i$ and let $S$ be an $R$-submodule of\/
	$T_1\oplus T_2$. Then $S$ has a~separated diagram $\pb{S;p}$.
\end{lemma}
\begin{proof}
	The~quotient $S_1 = S/P_2S$ is an~$R_1$-module with $r_1[s] := [(r_1,r_2)s]$ for any
	$(r_1,r_2)\in R$; the~action is well defined, since $(0,r)\in P_2$ acts trivially on
	$S_1$. In a~similar way we construct an~$R_2$-module $S_2$. We can see $S$ as a~submodule
	of $S_1\oplus S_2$, since $P_1S\cap P_2S=0$ reveals the~canonical homomorphism
	$S\to S_1\oplus S_2$ is injective.

	The~quotient group $\qnt S = S/(P_1S + P_2S)$ is an~$\qnt R$-module: an~element $\pbit
	{r_1,0,r_2}$ acts trivially on $\qnt S$. Because $P_iS\subset\ker(S\to\qnt S)$, there is
	an~induced surjective homomorphism $p_i\colon S_i\to\qnt S$. The~module $S$, when regarded
	as a~submodule of $S_1\oplus S_2$, consists of pairs $(s_1,s_2)$ of elements with the~same
	image in $\qnt S$. Indeed, $p_1([s_1]) = p_2([s_2])$ in $\qnt S$ implies
	\begin{equation}
		s_1-s_2 \in P_1S+P_2S \quad\Rightarrow\quad s_1 = s_2 + a_1s' + a_2s''
	\end{equation}
	for some elements $a_i\in P_i$ and $s',s''\in S$, so that the~pair $(s_1+P_2S, s_2+P_1S)$
	is an~image of $s_1-a_2s'' = s_2-a_1s' \in S$. Hence, $S$ is the~pullback of $\pb{S;p}$.
	Finally, $\ker p_i=P_iS_i$, since $p_i([s]) = 0$ holds if and only if $s\in P_1S+P_2S$.
\end{proof}

\begin{corollary}
	The~pullback of~\eqref{diag:separated-module} is a~separated module for any
	$R_i$-linear maps $p_i$.
\end{corollary}

\begin{corollary}\label{cor:separated-submodule}
	A~submodule of a~separated module is again separated. In particular, for a~linear map
	$f\colon M\to N$ we have the~following:
	\begin{itemize}
		\item if $M$ is separated, then $\ker f$ is separated, and
		\item if $N$ is separated, then $\im f$ is separated.
	\end{itemize}
\end{corollary}

\noindent
As we shall see later, separated modules are not closed under quotients. In fact, every
$R$-module is a~quotient a~separated module, see Section~\ref{sec:general-modules}.

We shall now proceed to showing that a~separated diagram \eqref{diag:separated-module} for
a~pullback module $S$ is unique. For that we need a~notion of a~separated homomorphism.

\begin{definition}
	A~homomorphism $f\colon S\to T$ between separated $R$-modules is \emph{separated}
	if it is induced by $R_i$-linear maps $f_i\colon S_i\to T_i$ for some subdirect sum
	presentations $S\subset S_1\oplus S_2$ and $T\subset T_1\oplus T_2$.
\end{definition}

\noindent
It is not difficult to see that both homomorphisms $f_i$ induce the~same $\bar f\colon\bar S
\to\bar T$, where $S = \left(\pb{S}\right)$ and $T = \left(\pb{T}\right)$. Hence, a~homomorphism
is separated, if and only if it is a~pullback of a~commuting diagram of linear maps
\begin{equation}
	\bfig
		\hsquares/->`<-`->`->`->`->`<-/<500,500,400>[%
			S_1`\qnt S`S_2`T_1`\qnt T`T_2;%
			p_1`p_2`f_1`\bar f`f_2`q_1`q_2%
		]
	\efig
\end{equation}

\begin{proposition}\label{prop:morphism-is-sep}
	Every map between separated modules is separated.
\end{proposition}
\begin{proof}
	Choose a~homomorphism of separated modules $f\colon S\to T$, and separated diagrams
	for $S$ and $T$. Then
	\begin{equation}\label{eq:separate-f}
		\begin{split}
			f\pbit{0,0,r_2s_2} &= \pbit{0,0,r_2}f\pbit{s} \\
												 &= \pbit{0,0,r_2}\pbit{t} = \pbit{0,0,r_2t_2},
		\end{split}
	\end{equation}
	where $s_1$ is chosen so that $p_1(s_1) = p_2(s_2) =: \bar s$. Hence, we can define $f_1$ as follows.
	For $s_1\in S_1$ choose $s_2\in S_2$ such that $\pbit{s}\in S$ and compute $f\pbit{s} = \pbit{t}$.
	Due to \eqref{eq:separate-f} the~element $t_1$ is independent of the~choice of $s_2$, and we can set
	$f_1(s_1):=t_1$. Define $f_2$ in a~similar way.
\end{proof}

\begin{remark}
	According to the~proof of Proposition~\ref{prop:morphism-is-sep}, both $f_1$ and $f_2$ are uniquely
	determined by $f$ if one chooses separated presentations of $S$ and $T$. Hence, the~assignment
	$f\mapsto (f_1,\bar f, f_2)$ is functorial.
\end{remark}

\begin{theorem}
	Every pullback module $S$ has a~unique separated diagram $\pb{S;p}$.
\end{theorem}
\begin{proof}
	The~existence is guaranteed by Lemma~\ref{lem:pullback-is-separated}. Given another presentation
	$\pb{T;q}$ we can separate the~identity homomorphism $\id\colon S\to S$ by Proposition~\ref
	{prop:morphism-is-sep} into $R_i$-linear maps $f_i\colon S_i\to T_i$ and $g_i\colon T_i\to S_i$.
	One can easily verify that $f_ig_i = \id_{T_i}$ and $g_if_i = \id_{S_i}$, either from functoriality
	or the~construction of these maps, which shows $S_i$'s are unique up to an~isomorphism. Hence, so
	is $\bar S = S_i/P_iS_i$ and the~theorem follows.
\end{proof}

\section{Morphisms of pullback modules}\label{sec:morphisms}

In this section we will characterize basic properties of homomorphisms of pullback modules
in terms of maps between their pullback diagrams. Hereafter in this section a~module $M$
comes always with a~fixed pullback diagram
\begin{equation}
	\pb{M;p},
\end{equation}
where $M_i$ is an~$R_i$-module and $\qnt M$ is an~$\qnt R$-module, and a~morphism
$f\colon M\to N$ is given by a~triple $(f_1,\bar f,f_2)$ fitting into a~commuting
diagram
\begin{equation}
	\bfig
		\hsquares/->`<-`->`->`->`->`<-/<500,500,400>[%
			M_1`\qnt M`M_2`N_1`\qnt N`N_2;%
			p_1`p_2`f_1`\bar f`f_2`q_1`q_2%
		]
	\efig
\end{equation}

\begin{proposition}\label{prop:sep-monomorphism}
	A~homomorphism $f\colon M\to N$ of pullback modules is injective if and only if the~map
	$\mu\colon\ker f_1\oplus\ker f_2\to\qnt M$, $(m_1,m_2) \mapsto p_1(m_1)-p_2(m_2)$, is
	a~monomorphism.
\end{proposition}
\begin{proof}
	Assume $\mu$ is injective. If $f\pbit{m_1,\overline m,m_2}=0$, then $m_i\in\ker f_i$ and
	$\mu(m_1,m_2)=\overline m-\overline m = 0$. This shows $m_i=0$.

	Conversely, choose $m_i\in\ker f_i$ such that $\mu(m_1,m_2)=0$. This forces
	$p_1(m_1)=p_2(m_2)$, so that $(m_1,m_2)\in S$. However, $f\pbit{m_1,\overline m,m_2} = 0$
	implies $m_i=0$, which proves injectivity of $\mu$.
\end{proof}

\begin{corollary}\label{cor:mono-sep-cond}
	A~homomorphism of pullback modules $f\colon M\to N$ is a~monomorphism if and only if
	\begin{enumerate}
		\item $\ker f_i\cap \ker p_i = 0$, and
		\item $p_1(\ker f_1)\cap p_2(\ker f_2) = 0$.
	\end{enumerate}
	If $\pb{M;p}$ is a~separated diagram, the~first condition translates into
	$\ker f_i\cap P_iM_i=0$.
\end{corollary}

Both the~kernel and the~pullback are categorical limits, so that the~kernel of a~separated
homomorphism $f$ is given by the~pullback of kernels of its components i.e. the~left column
below is a~pullback diagram of $\ker f$:
\begin{equation}
	\bfig
		\iiixiii|aaaaaalrrlrr|/->`->`->`->`->`->`->`->`->`<-`<-`<-/<500,400>[%
			\ker f_1`M_1`N_1`\ker\bar f`\qnt M`\qnt N`\ker f_2`M_2`N_2;%
			`f_1``\bar f``f_2`c_1`p_1`q_1`c_2`p_2`q_2%
		]
	\efig
\end{equation}
The~homomorphisms $c_i$ exist and are unique due to the~universal property of the~kernel.
However, they are usually not surjective, so the~left column is seldom a~separated
diagram for $\ker f$.

Assume now $N$ is given by a~preseparated diagram, i.e. a~pullback of epimorphisms. If
$f\colon M\to N$ is surjective, so must be each $f_i\colon M_i\to N_i$ (because $N$ is
a~subdirect sum in $N_1\oplus N_2$). The converse does not hold in general. The~proposition
below is not the~weakest possible statement, but it shows what difficulties may occur.
One condition that obviously can be weakened is the~surjectivity of $f_i$ or $\bar f$:
all we use in the~proof below is that their images contain certain submodules of $N_i$
or $\bar N$.

\begin{proposition}\label{prop:sep-epimorphism}
	A~homomorphism $f\colon M\to N$ is an~epimorphism
	if any of the~conditions below holds:
	\begin{enumerate}
	\item both $f_1, f_2$ and one of $c_i\colon\ker f_i\to\ker\bar f$ are epimorphisms, or
	\item $f_1$ is an~epimorphism and the~canonical map from $M_2$ to the~pullback of
				$\pb{\qnt M,\qnt N,N_2}$ is surjective, or
	\item $f_2$ is an~epimorphism and the~canonical map from $M_1$ to the~pullback of
				$\pb{\qnt M,\qnt N,N_1}$ is surjective, or
	\item $\bar f$ is an~epimorphism and the~canonical map from $M_i$ to the~pullback of
				$\pb{\qnt M,\qnt N,N_i}$ is surjective, where $i=1,2$.
	\end{enumerate}
\end{proposition}
\begin{proof}
	Pick an element $\pbit{n}\in N$. Surjectivity under condition (1) is proved by an~easy
	diagram chase: first for $i=1,2$ find $m_i\in M_i$ such that $f_i(m_i)=n_i$ and then
	use surjectivity of $c_1$ or $c_2$ to modify one of these elements, so that they project
	on the~same $\overline m$. More precisely, if $c_i$ is surjective, add to $m_i$ an~element
	from $c_i^{-1}(p_1(m_1)-p_2(m_2))$.

	For (2), due to surjectivity of $f_1$ we have $n_1=f_1(m_1)$ for some $m_1\in M_1$. Then
	$(p_1(m_1),n_2)$ is an~element of the~pullback of $\pb{\qnt M,\qnt N,N_2}$, and as such
	it comes from some $m_2\in M_2$. Then $\pbit{m_1,\overline m,m_2}$ in an~element of $M$
	sent by $f$ to $\pbit{n}$. In a~similar way we prove (3).

	For last condition, pick $\bar m$ that covers $\bar n$. Then we have elements $m_i\in M_i$,
	$i=1,2$, such that $p_i(m_i)=\overline m$, and $f_i(m_i)=n_i$.
\end{proof}

\begin{corollary}
	Suppose there is a~diagram with exact rows,
	\begin{equation}
		\bfig
			\iiixiii|aaaaaalrrlrr|/->`->`->`->`->`->`->`->`->`<-`<-`<-/<500,400>{63}[%
				K_1`M_1`N_1`\qnt K`\qnt M`\qnt N`K_2`M_2`N_2;```````````%
		]
		\efig
	\end{equation}
	where the vertical maps in the~first column are epimorphisms. Then the induced sequence
	of $R$-modules $0\longrightarrow K\longrightarrow M\longrightarrow N\longrightarrow 0$
	is exact.
\end{corollary}
\begin{proof}
	The~pullback functor is a~limit, and as such it is left exact. The~additional condition
	on the~first column guarantees that $M\longrightarrow N$ is surjective.
\end{proof}

\section{Separated presentations and \texorpdfstring{$R$}{R}-diagrams}\label{sec:general-modules}

\begin{definition}
	A~homomorphism $f\colon K\to M$ of pullback modules is called a~\emph{pullback presentation}
	of a~module $M$, if it is a~pullback of a~commuting diagram
	\begin{equation}\label{diag:separated-presentation}
		\bfig
			\hsquares|aalmraa|/->`<-`->`->`->`->`<-/<500,500,400>[%
				K_1`\qnt K`K_2`S_1`\qnt S`S_2;%
				q_1`q_2`f_1`\bar f`f_2`p_1`p_2%
			]
		\efig
	\end{equation}
	and $M=\coker f$. We say $f$ is a~\emph{separated presentation}\footnote{
		This is different from a~separated \emph{representation} defined in \cite{Mixed ZG modules}
		as an~epimorphisms $S\to M$, minimal in some sense, where $S$ is separated.
	} of $M$, if both $K$ and $S$ are separated.
\end{definition}

In the~view of Proposition~\ref{prop:morphism-is-sep} every homomorphisms $f\colon K\to S$ of
separated modules is induced by a~unique diagram of the~form \eqref{diag:separated-presentation}.
In particular, one has $\ker q_i=P_iK_i$ and $\ker p_i=P_iS_i$. Furthermore, due to Corollary~%
\ref{cor:separated-submodule} we can assume $f\colon K\to S$ is injective.

\begin{proposition}\label{prop:sep-prep-exists}
	Every $R$-module $M$ has a~separated presentation.
\end{proposition}
\begin{proof}
	The~ring $R$, regarded as an~$R$-module, is separated, so is every free $R$-module.
\end{proof}

\noindent
More can be shown if $\qnt R$ is semisimple.

\begin{proposition}\label{prop:R-diagrams}
	Assume that $\qnt R$ is semisimple. Then every $R$-module $M$ has a~separated presentation
	$f\colon K\to S$ such that
	\begin{enumerate}
		\item $K$ is the~pullback of $\pb{K,K,K;\id,\id}$, and
		\item both $f_1$ and $f_2$ are monomorphisms, and $\bar f = 0$.
	\end{enumerate}
\end{proposition}

\noindent
A~presentation of this type is called an~\emph{$R$-diagram} \cite{Mixed ZG modules} and is
written as a~diagram of four morphisms
\begin{equation}
	\bfig
		\Atrianglepair/->``->`->>`<<-/<500,300>[K`S_1`\qnt S`S_2;q_1``q_2`p_1`p_2]
	\efig
\end{equation}
Existence of such a~presentation was first proven in \cite{Mixed ZG modules} under the~assumption
that both $R_i$ are Dedekind domains and that $\qnt R$ is a~field. Our goal for this section is to
describe an~algorithm computing an~$R$-diagram from a~given separated presentation of $M$, proving
Proposition \ref{prop:R-diagrams}. We begin with the~following technical lemma.

\begin{lemma}\label{lem:exactness}
	Suppose there is a~commutative diagram with exact rows
	\begin{equation}
		\bfig
			\iiixiii|aaaaaallllll|/->`->`->`->`->`->`->`->`->`<-`<-`<-/<500,400>{23}[%
					L_1`M_1`N_1`\qnt L`\qnt M`\qnt N`L_2`M_2`N_2;%
					f_1`g_1`\bar f`\bar g`f_2`g_2`c_1```c_2``%
			]
		\efig
	\end{equation}
	where the maps $c_i$ are epimorphisms. Then the induced sequence
	$L\stackrel{f}\longrightarrow M\stackrel{g}\longrightarrow N\longrightarrow 0$
	is exact. Moreover, if $\pb{M}$ is separated, so is $\pb{N}$.
\end{lemma}
\begin{proof}
	The~surjectivity of $c_i$ implies $g$ is an~epimorphism. Obviously, $g\circ f = 0$ and to 
	show that $\ker g=\im f$, pick any element $\pbit{m_1,\overline m,m_2}$ from $\ker g$. By
	the~exactness of rows $m_i=f_i(k_i)$ for some $k_i\in L_i$, and $\overline m=\bar f(\bar k)$
	for some $\bar k\in\qnt L$. Since $\bar f$ is a~monomorphism and $\bar f(c_i(k_i)) =
	\overline m = \bar f(\bar k)$, the~triple $\pbit{k}$ is an~element of $L$ and $f(k) = m$.

	For the~last statement notice that the~Snake Lemma implies
	\begin{equation}
		\ker(N_i\to\qnt N) = g_i(\ker(M_i\to\qnt M)) = g_i(P_iM_i) = P_iN_i,
	\end{equation}
	and the~surjectivity of $M_i\to\qnt M$ implies $N_i\to\qnt N$ are surjective as well.
\end{proof}

The meaning of this lemma is that some quotients of separated modules are still separated.
This allows us sometimes to reduce a~given separated presentation of a~module $M$. Hereafter
fix a~pullback presentation $f\colon K\to S$ of a~module $M$.

\begin{lemma}\label{lem:sep-pres-reduce}
	Let $L=(\pb{L;u})$ be a~submodule of $K$ such that each $u_i$ is surjective and $\bar
	f|_{\bar L}$ is injective. Then $K/L \to S/f(L)$ is a~pullback presentation of $M$.
	Furthermore, if $K$ and $S$ are separated, so are $K/L$ and $S/f(L)$.
\end{lemma}
\begin{proof}
	The~cokernel, as a~colimit, is a~right exact functor, so that the~bottom row in the~diagram
	below is exact:
	\begin{equation}
		\bfig
			\iiixiii|aaaaaarrrrrr|<500,400>{7}[%
				L`L`0`K`S`M`K/L`S/f(L)`M;\id``f`````f|_L````\id%
			]
		\efig
	\end{equation}
	Lemma~\ref{lem:exactness} guarantees that both $K/L$ and $S/f(L)$ are pullback modules.
	If both $K$ and $S$ are separated, so are their quotients.
\end{proof}

\begin{example}\label{ex:reduction-of-K}
	Choose a~separated presentation $f\colon K\to S$ of $M$ and let $L$ be the~submodule of
	$K$ given by a~diagram $\pb{\ker q_1,0,\ker q_2}$. Then the~quotient presentation is
	again separated, and $K' = (\pb{\qnt K,\qnt K,\qnt K;\id,\di})$.
\end{example}

\begin{example}\label{ex:reduction-of-bar-f}
	Suppose we have a~separated presentation $f\colon K\to S$ of $M$ such that $\ker\bar f$ is
	a~direct summand of $\qnt K$ with a~complement $\qnt L$. Then $L = (\pb{q_1^{-1}(\qnt L),
	\qnt L,q_2^{-1}(\qnt L)})$ satisfies the~assumptions of Lemma~\ref{lem:sep-pres-reduce} and
	the~quotient presentation $K'\to S'$ is separated, with $\bar f' = 0$.
\end{example}

We will now demonstrate how to obtain an $R$-diagram for $M$ from any separated presentation.
Starting with a~given presentation, we can already apply the~reductions from the~two examples
above, so that $K_1=\bar K=K_2$ and $\bar f=0$. It remains to make $f_1$ and $f_2$ injective,
for which we shall apply another type of reduction.

\begin{lemma}
	Let $f\colon L\to M$ be a~morphism of pullback modules, such that $f_2=0$, $\bar f=0$ and
	the~homomorphism $L_2\to\qnt L$ is surjective. Then $M' := M/f(L)$ is a~pullback module
	with a~diagram $(\pb{M_1/f_1(L_1),\qnt M,M_2})$. Moreover, if $M$ is separated, so is $M'$.
\end{lemma}
\begin{proof}
	Consider a~commutative diagram
	\begin{equation}
		\bfig
			\iiixiii|aaaaaallllll|/->`->`->`->`->`->`->`->`->`<-`<-`<-/<600,400>{7}[%
				L_1`M_1`M_1/f_1(L_1)`\qnt L`\qnt M`\qnt M`L_2`M_2`M_2;%
				f_1`g_1`0`\id`0`\id`c_1```c_2``%
			]
		\efig
	\end{equation}
	with the~top row exact. We want to show that it induces an~exact sequence
	\begin{equation}
		L\to^f M \to^g N\to 0,
	\end{equation}
	where $N$ is the~pullback of the~last column. First, the~map $g$ is surjective by
	Proposition~\ref{prop:sep-epimorphism} (since $\ker g_2=\ker\bar g=0$), and clearly
	$g\circ f=0$. To show that $\ker g=\mathrm{im}f$ pick an~element $\pbit{m_1,\overline
	m,m_2}$ from $\ker g$ and choose $l_1\in L_1$, for which $f_1(l_1) = m_1$. Let $\bar l$
	be its image in $\qnt L$, and choose any $l_2\in L_2$ which projects to $\bar l$. Then
	$l=\pbit{l}$ is an~element of $L$ such that $f(l) = m$.
\end{proof}

\begin{corollary}\label{cor:sep-pres-red-two}
	Let $L=(\pb{L;u})$ be a~pullback submodule of $K$ such that each $u_2$ is surjective
	and $f_2(L_2) = \bar f(\qnt L)=0$. Then $K/L \to S/f(L)$ is a~pullback presentation of
	$M$. Furthermore, if $K$ and $S$ are separated, so are $K/L$ and $S/f(L)$.
\end{corollary}

\begin{example}\label{ex:reduction-to-monos}
	Given a~separated presentation $f\colon K\to S$ of $M$ with $\bar f=0$, we can modify
	it into $f'\colon K'\to S'$ with both $f_i$ being injective. Namely, take as $L$ the~%
	pullback submodule of $K$ with a~diagram
	\begin{equation}
		\pb{q_1^{-1}(q_2(L_2)),q_2(L_2),L_2},
	\end{equation}
	where $L_2 = \ker f_2$. Since it satisfies the~assumptions of Corollary \ref{cor:sep-pres-red-two},
	we can form a~quotient presentation $f''\colon K''\to S''$, with an~injective $f''_2$. We repeat
	this for $f_1$, obtaining $f'\colon K'\to S'$. Notice that injectivity of $f''_2$ guarantees
	injectivity of $f'_2$.
\end{example}

We will now combine all the~reductions together to obtain an $R$-diagram from a~given separated
presentation of a~module $M$, proving Proposition \ref{prop:R-diagrams}.

\begin{proof}[Proof of Proposition \ref{prop:R-diagrams}]
	Choose a~separated presentation $f\colon K\to S$ of an~$R$-module $M$ --- it exists due to
	Proposition~\ref{prop:sep-prep-exists} --- and set $T_i := \ker f_i$ and $\qnt T_i := q_i(T_i)$.
	Since the~ring $\qnt R$ is semisimple, there is a~submodule $U\subset\qnt K$ such that $\qnt K
	= \ker\bar f\oplus U$. Consider the~submodule $L$ of $K$ given by a~diagram
	\begin{equation}
		\pb{q_1^{-1}(U+\qnt T_2)+T_1 , U+\qnt T_1+\qnt T_2 , q_2^{-1}(U+\qnt T_1)+T_2}
	\end{equation}
	Then $K/L\to S/f(L)$ is an~$R$-diagram for $M$. Indeed, the~same quotient can be obtained
	by applying the~three reductions from Examples \ref{ex:reduction-of-K},
	\ref{ex:reduction-of-bar-f}, and \ref{ex:reduction-to-monos}.
\end{proof}

\section{\texorpdfstring{$R$}{R}-diagrams for homology over \texorpdfstring{$p$}{p}-pullback rings}
\label{sec:homology}

We will now apply the~results of the~previous section to compute $R$-diagrams for homology
modules of a~chain complex $(C,\diff)$ of free modules over a~$p$-pullback ring $R$, defined
by a~pullback diagram
\begin{equation}
	\pb{\Z,\Z_p,\Z},
\end{equation}
where $p$ is a~prime number. Clearly, $P_1\subset R$ is generated by $(p,0)$ and $P_2$ by
$(0,p)$. Because $\Z_p$ is a~field, the~Proposition~\ref{prop:R-diagrams} holds, and it makes
sense to ask for $R$-diagrams for homology modules.

First, free modules have natural separated diagrams. Hence, we can assume the~differential
$\diff$ is given by commutative diagrams like the~one below
\begin{equation}
	\bfig
		\hsquares|aallraa|/->`<-`->`->`->`->`<-/<500,500,400>[%
			\Z^m`\Z_p^m`\Z^m`\Z^n`\Z_p^n`\Z^n;``\diff_1`{\bar\diff\rule[-1.25ex]{0pt}{1.25ex}}`\diff_2``%
		]
	\efig
\end{equation}
where the~horizontal arrows are direct powers of the~projections $R_i\to\qnt R$. The~naive
pullback presentation of $\ker\diff$ is given by
\begin{equation}
	\pb{\ker \diff_1,\ker\bar\diff,\ker\diff_2; q}.
\end{equation}
It is not separated in general, as $q_1$ and $q_2$ are rarely surjective, and usually
$\ker q_i\neq p\ker\diff_i$. We start with a~technical lemma to fix this.

\begin{lemma}
	Choose two homomorphisms $f,g\colon F\to F'$ between free abelian groups, and let $K:=
	\ker f\cap\ker g$. Then there exists a~subgroup $U\subset\ker f$ such that $\ker f = K\oplus U$.
\end{lemma}
\begin{proof}
	The~group $K$ is free as a~subgroup of $F$, and there is an~isomorphism
	\begin{equation}
		\quotient{\ker f}{K}\cong\im\left(g|_{\ker f}\right),
	\end{equation}
	showing that the~quotient is a~free group (since the~image is a~subgroup of the~free group
	$F'$). Hence, it is isomorphic to a~subgroup $U$ in $\ker f$, complementary to $K$.
\end{proof}

\noindent
Applying this lemma to $\diff_1$ and $\diff_2$ we can construct five sets of generators:
\begin{enumerate}
	\item $\left\{v_r^{12}\right\}$ is a~basis of $\ker\diff_1\cap\ker\diff_2$,
	\item $\left\{v_r^1\right\}$ is a~basis of a~complement of $\ker\diff_1\cap\ker\diff_2$ in
				$\ker\diff_1$,
	\item $\left\{v_r^2\right\}$ is a~basis of a~complement of $\ker\diff_1\cap\ker\diff_2$ in
				$\ker\diff_2$,
	\item $\left\{\bar v_r\right\}$ is a~basis of a~complement of $q_1(\ker\diff_1)+q_2(\ker\diff_2)$
				in $\ker\bar\diff$, and
	\item $\left\{\bar v_r^c\right\}$ is a~basis of a~complement of $\ker\bar\diff$ in $\Z_p^m$.
\end{enumerate}
The~last two exist, since we look for complements of vector subspaces.

\begin{proposition}
	Given bases $\{v^1_r\}$, $\{v^2_r\}$, and $\{v^{12}_r\}$ as above define
	\begin{itemize}
		\item $Q_1  = \Z\langle v^{12}_r,pv^1_{r'}\rangle \oplus \Z_p\langle pv^2_{r''}\rangle$,
		\item $Q_2  = \Z\langle v^{12}_r,pv^2_{r'}\rangle \oplus \Z_p\langle pv^1_{r''}\rangle$,
		\item $\qnt Q = \Z_p\langle v^{12}_r,pv^1_{r'},pv^2_{r''}\rangle$.
	\end{itemize}
	Then $\pb{Q}$ is a~canonical separated presentation of\/ $\ker\diff$, where each map sends
	a~generator of\/ $Q_i$ to the~corresponding generator of\/ $\qnt Q$.
\end{proposition}
\begin{proof}
	Obviously, each map $Q_i\to\qnt Q$ is surjective and $\ker(Q_i\to\qnt Q) = pQ_i$.
	It remains to show $Q = \ker\diff$. First, $\pb{Q}$ is a~direct sum of the~following
	diagrams
	\begin{gather*}
		\pb{\Z\langle v^{12}_r\rangle,\Z_p\langle v^{12}_r\rangle,\Z\langle v^{12}_r\rangle}\\
		\pb{\Z\langle pv^1_r\rangle,\Z_p\langle pv^1_r\rangle,\Z_p\langle pv^1_r\rangle}\\
		\pb{\Z_p\langle pv^2_r\rangle,\Z_p\langle pv^2_r\rangle,\Z\langle pv^2_r\rangle}
	\end{gather*}
	resulting in a~submodule of $R^m$ generated by elements $\pbit{v^{12}_r,v^{12}_r,v^{12}_r}$,
	$\pbit{pv^1_r,0,0}$ and $\pbit{0,0,pv^2_r}$. These are precisely the~generators of $\ker\diff$,
	which ends the~proof.
\end{proof}

\indent
To compute a~separated presentation of $H^n$, it remains to rewrite the~components of the~differential
$\diff\colon C^{n-1}\to Q\subset C^n$ using the~presentation of $Q$ given above. Since $\diff_i$ must
take values in the~free part of $Q_i$, this is a~simple problem from linear algebra.\footnote{
	For instance, to find a~projection $\dom\diff_i\to Q_i$ one can extend the~basis $\{v_r^{12},
	v_{r'}^i\}$ of $\ker\diff_i$ into a~basis of $\dom\diff_i$ and inverse the~matrix having these
	vectors as columns.
}

The~above results in a~separated presentation of $H^n$ with a~diagram
\begin{equation}
	\bfig
		\hsquares|aallraa|/->`<-`->`->`->`->`<-/<500,500,400>[%
			\Z^\ell`\Z_p^\ell`\Z^\ell`Q_1`\qnt Q`Q_2;%
			``\diff_1`{\bar\diff\rule[-0.9ex]{0pt}{0.9ex}}`\diff_2``%
		]
	\efig
\end{equation}

The~next step is two apply the~three reductions from Section~\ref{sec:general-modules}. First,
we compute the~five sets of generators as before, denoting them by $\{\bar w_r\}$, $\{\bar w^c_r\}$,
etc. According to the~proof of Proposition~\ref{prop:R-diagrams}, the~module $K$ is a~vector
space over $\Z_p$ with the~basis $\{\bar w_r\}$, while the~separated module $S$ has components
\begin{align*}
	\qnt S &= \quotient{\Z_p\langle v_r^{12},pv_{r'}^1,pv_{r''}^2\rangle}{\im\bar\diff}	\\
	S_1 &=	\Z_p\langle pv_{r''}^2\rangle	\oplus
					\quotient{\Z\langle v_r^{12},pv_{r'}^1\rangle}{\diff_1(\langle w^c_s, w^2_r\rangle)+p\im\diff_1}
					\\
	S_2 &=	\Z_p\langle pv_{r''}^1\rangle	\oplus
					\quotient{\Z\langle v_r^{12},pv_{r'}^2\rangle}{\diff_2(\langle w^c_s, w^1_r\rangle)+p\im\diff_2}
					\oplus\Z_p\langle pv_{r''}^1\rangle
\end{align*}
where $w^c_s$ is the~element $\bar w^c_s$ regarded with integral coefficients. Indeed,
the~subspace $U\subset\Z^\ell$ from the~proof of Proposition~\ref{prop:R-diagrams} is
generated by vectors $\bar w^c_s$, whereas $T_i\subset\Z^\ell$ is generated by $v^{12}_r$
and $v^i_{r'}$. Therefore,
\begin{align}
	\bar\diff(U+\qnt T_1+\qnt T_2) &= \bar\diff U = \im\bar\diff,\\
	\label{eq:diff1(L1)}
	\diff_1(q_1^{-1}(U+\qnt T_2)+T_1) &= \diff_1(\langle w^c_s,w^2_r\rangle)+p\im\diff_1,\\
	\label{eq:diff2(L2)}
	\diff_2(q_2^{-1}(U+\qnt T_1)+T_2) &= \diff_2(\langle w^c_s,w^1_r\rangle)+p\im\diff_2.
\end{align}
Since $0 = \diff_2(w^2_r) \equiv \diff_1(w^2_r)\,(\mathrm{mod}\,p)$, the~latter is
divisible by $p$, and similarly for $\diff_2(w^1_r)$. Finally, the~homomorphisms $K\to S_i$ are
induced by $\diff_i$: an~element $\bar w_r$ is sent to $\diff_i(w_r)$, where $w_r$ is the~element
$\bar w_r$ regarded with integral coefficients.

\end{document}